\documentclass[12pt,reqno,a4paper]{amsart}%
\usepackage{amsfonts}
\usepackage{amsmath}
\usepackage{amssymb}
\usepackage{graphicx}%
\newtheorem{theorem}{Theorem}[section]
\newtheorem{corollary}[theorem]{Corollary}

\theoremstyle{definition}
\newtheorem{definition}[theorem]{Definition}
\newtheorem{remark}[theorem]{Remark}
\numberwithin{equation}{section} 
\begin{document}
\title[Faber polynomial coefficient estimates ... ]{Faber polynomial coefficient estimates for a class of analytic bi-univalent
functions }
\author{A.A. AMOURAH}
\address{A.A. AMOURAH: Department of Mathematics, Faculty of Science and Technology,
Irbid National University, Irbid, Jordan.}
\email{alaammour@yahoo.com.}

\begin{abstract}
In the present paper, we were mainly concerned with obtaining estimates for the general Taylor-Maclaurin coefficients for functions in a certain general subclass of analytic bi-univalent functions. For this purpose, we used the Faber polynomial expansions. Several connections to some of the earlier known results are also pointed out.
\end{abstract}
\maketitle

\section{Introduction}
Let $\mathcal{A}$ denote the class of all analytic functions $f$ defined in
the open unit disk $\mathbb{U}=\{z\in\mathbb{C}:\left\vert z\right\vert <1\}$
and normalized by the conditions $f(0)=0$ and $f^{\prime}(0)=1$. Thus each
$f\in\mathcal{A}$ has a Taylor-Maclaurin series expansion of the form:%

\begin{equation}
f(z)=z+\sum\limits_{n=2}^{\infty}a_{n}z^{n},\ \ (z\in\mathbb{U}).\text{
\ \ \ \ } \label{ieq1}
\end{equation}

Further, let $\mathcal{S}$ denote the class of all functions $f\in\mathcal{A}$
which are univalent in $\mathbb{U}$ (for details, see \cite{Duren}; see also
some of the recent investigations \cite{C1,C4,C5,C6,C2}).

Two of the important and well-investigated subclasses of the analytic and
univalent function class $\mathcal{S}$ are the class $\mathcal{S}^{\ast
}(\alpha)$ of starlike functions of order $\alpha$ in $\mathbb{U}$ and the
class $\mathcal{K}(\alpha)$ of convex functions of order $\alpha$ in
$\mathbb{U}$. By definition, we have%

\begin{equation}
\mathcal{S}^{\ast}(\alpha):=\left\{  f:\ f\in\mathcal{S}\ \ \text{and}%
\ \ \mbox{Re}\left\{  \frac{zf^{\prime}(z)}{f(z)}\right\}  >\alpha,\quad
(z\in\mathbb{U};0\leq\alpha<1)\right\}, \label{d1}%
\end{equation}
and%
\begin{equation}
\mathcal{K}(\alpha):=\left\{  f:\ f\in\mathcal{S}\ \ \text{and}%
\ \ \mbox{Re}\left\{  1+\frac{zf^{\prime\prime}(z)}{f^{\prime}(z)}\right\}
>\alpha,\quad(z\in\mathbb{U};0\leq\alpha<1)\right\}. \label{d2}%
\end{equation}

It is clear from the definitions (\ref{d1}) and (\ref{d2}) that
$\mathcal{K}(\alpha)\subset\mathcal{S}^{\ast}(\alpha)$. Also we have%

\[
f(z)\in\mathcal{K}(\alpha)\ \ \text{iff}\ \ zf^{\prime}(z)\in\mathcal{S}%
^{\ast}(\alpha),
\]
and%
\[
f(z)\in\mathcal{S}^{\ast}(\alpha)\ \ \text{iff}\ \ \int_{0}^{z}\frac{f(t)}%
{t}dt=F(z)\in\mathcal{K}(\alpha).
\]

It is well-known that, if $f(z)$ is an univalent analytic function
from a domain $\mathbb{D}_{1}$ onto a domain $\mathbb{D}_{2}$, then the
inverse function $g(z)$ defined by%
\[
g\left(  f(z)\right)  =z,\ \ (z\in\mathbb{D}_{1}),
\]
is an analytic and univalent mapping from $\mathbb{D}_{2}$ to
$\mathbb{D}_{1}$. Moreover, by the familiar Koebe one-quarter theorem (for
details, (see \cite{Duren}), we know that the image of $\mathbb{U}$ under
every function $f\in\mathcal{S}$ contains a disk of radius $\frac{1}{4}$.

According to this, every function $f\in\mathcal{S}$ has an inverse map
$f^{-1}$ that satisfies the following conditions:

\begin{center}
$f^{-1}(f(z))=z \ \ \ (z\in\mathbb{U}),$
\end{center}
and
\begin{center}
$f\left(  f^{-1}(w)\right)  =w$ $\ \ \ \left(  |w|<r_{0}(f);r_{0}(f)\geq
\frac{1}{4}\right)  $.
\end{center}

In fact, the inverse function is given by%
\begin{equation}
f^{-1}(w)=w-a_{2}w^{2}+(2a_{2}^{2}-a_{3})w^{3}-(5a_{2}^{3}-5a_{2}a_{3}%
+a_{4})w^{4}+\cdots. \label{ieq2}%
\end{equation}

A function $f\in\mathcal{A}$ is said to be bi-univalent in
$\mathbb{U}$ if both $f(z)$ and $f^{-1}(z)$ are univalent in $\mathbb{U}$. Let
$\Sigma$ denote the class of bi-univalent functions in $\mathbb{U}$ given by
(\ref{ieq1}). Examples of functions in the class $\Sigma$ are%
\[
\frac{z}{1-z},\ -\log(1-z),\ \frac{1}{2}\log\left(  \frac{1+z}{1-z}\right)
,\cdots.
\]

It is worth noting that the familiar Koebe function is not a member
of $\Sigma$, since it maps the unit disk $\mathbb{U}$ univalently onto the
entire complex plane except the part of the negative real axis from $-1/4$ to
$-\infty$. Thus, clearly, the image of the domain does not contain the unit
disk $\mathbb{U}$. For a brief history and some intriguing examples of
functions and characterization of the class $\Sigma$, see Srivastava et al.
\cite{C27}, Frasin and Aouf \cite{C28}, and Yousef et al. \cite{C3}.

In 1967, Lewin \cite{C21} investigated the bi-univalent function class
$\Sigma$ and showed that $|a_{2}|<1.51$. Subsequently, Brannan and Clunie
\cite{C22} conjectured that $|a_{2}|\leq\sqrt{2}.$ On the other hand,
Netanyahu \cite{C23} showed that $\underset{f\in\Sigma}{\max}$ $|a_{2}%
|=\frac{4}{3}.$ The best known estimate for functions in $\Sigma$ has been
obtained in 1984 by Tan \cite{C24}, that is, $|a_{2}|<1.485$. The coefficient
estimate problem for each of the following Taylor-Maclaurin coefficients
$|a_{n}|$ $(n\in\mathbb{N}\backslash\{1,2\})$ for each $f\in\Sigma$ given by
(\ref{ieq1}) is presumably still an open problem.

In this paper, we use the Faber polynomial expansions for a general
subclass of analytic bi-univalent functions to determine estimates for the
general coefficient bounds $|a_{n}|$.

The Faber polynomials introduced by Faber \cite{12} play an important role in
various areas of mathematical sciences, especially in geometric function
theory. The recent publications \cite{15} and \cite{17} applying the Faber
polynomial expansions to meromorphic bi-univalent functions motivated us to
apply this technique to classes of analytic bi-univalent functions. In the
literature, there are only a few works determining the general coefficient
bounds $|a_{n}|$ for the analytic bi-univalent functions given by (\ref{ieq1}) using Faber
polynomial expansions (see for example, \cite{16,20,19}). Hamidi and Jahangiri
\cite{16} considered the class of analytic bi-close-to-convex functions.
Jahangiri and Hamidi \cite{19} considered the class defined by Frasin and Aouf
\cite{C28}, and Jahangiri et al. \cite{20} considered the class of analytic
bi-univalent functions with positive real-part derivatives.

\section{\bigskip The class $\mathfrak{B}_{\Sigma}^{\mu}(\alpha,\lambda,\delta
)$}
Firstly, we consider a comprehensive class of analytic bi-univalent functions
introduced and studied by Yousef et al. \cite{class} defined as follows:

\begin{definition}
\label{def22} (See \cite{class}) For $\lambda\geq1,$ $\mu\geq0,$ $\delta\geq0$
and $0\leq\alpha<1$, a function $f\in\Sigma$ given by (\ref{ieq1}) is said to
be in the class $\mathfrak{B}_{\Sigma}^{\mu}(\alpha,\lambda,\delta)$ if the
following conditions hold for all $z,w\in\mathbb{U}$:
\begin{equation}
\mbox{Re}\left(  (1-\lambda)\left(  \frac{f(z)}{z}\right)  ^{\mu}+\lambda
f^{\prime}(z)\left(  \frac{f(z)}{z}\right)  ^{\mu-1}+\xi\delta zf^{\prime
\prime}(z)\right)  >\alpha\label{ieq23}%
\end{equation}
and%
\begin{equation}
\mbox{Re}\left(  (1-\lambda)\left(  \frac{g(w)}{w}\right)  ^{\mu}+\lambda
g^{\prime}(w)\left(  \frac{g(w)}{w}\right)  ^{\mu-1}+\xi\delta wg^{\prime
\prime}(w)\right)  >\alpha, \label{ieq24}%
\end{equation}
where the function $g(w)=f^{-1}(w)$ is defined by (\ref{ieq2}) and $\xi
=\frac{2\lambda+\mu}{2\lambda+1}$.
\end{definition}

\bigskip
\begin{remark}
\label{rem2} In the following special cases of Definition \ref{def22}; we show how the
class of analytic bi-univalent functions $\mathfrak{B}_{\Sigma}^{\mu}%
(\alpha,\lambda,\delta)$ for suitable choices of $\lambda$, $\mu$ and $\delta$
lead to certain new as well as known classes of analytic bi-univalent
functions studied earlier in the literature.

(i) For $\delta=0,$ we obtain the bi-univalent function class
$\mathfrak{B}_{\Sigma}^{\mu}(\alpha,\lambda,0):=\mathfrak{B}_{\Sigma}^{\mu}%
(\alpha,\lambda)$ introduced by \c{C}a\u{g}lar et al. \cite{C29}. \vspace{0.05in}

(ii) For $\delta=0$ and $\mu=1,$ we obtain the bi-univalent function class
$\mathfrak{B}_{\Sigma}^{1}(\alpha,\lambda,0):=\mathfrak{B}_{\Sigma}(\alpha
,\lambda)$ introduced by Frasin and Aouf \cite{C28}. \vspace{0.05in}

(iii) For $\delta=0$, $\mu=1$, and $\lambda=1,$ we obtain the bi-univalent function class
$\mathfrak{B}_{\Sigma}^{1}(\alpha,1,0):=\mathfrak{B}_{\Sigma}(\alpha)$ introduced by Srivastava et al. \cite{C27}. \vspace{0.05in}

(iv) For $\delta=0$, $\mu=0$, and $\lambda=1,$ we obtain the well-known class $\mathfrak{B}_{\Sigma }^{0}(\alpha,1,0):=\mathcal{S}^*_\Sigma(\alpha)$ of bi-starlike functions of order $\alpha$. \vspace{0.05in}

(iv) For $\mu=1$, we obtain the well-known class $\mathfrak{B}_{\Sigma }^{1}(\alpha,\lambda,\delta):=\mathfrak{B}_{\Sigma }(\alpha,\lambda,\delta)$ of bi-univalent functions.
\end{remark}

\newpage

\section{Coefficient estimates}
Using the Faber polynomial expansion of functions $f\in\mathcal{A}$
of the form (\ref{ieq1}), the coefficients of its inverse map $g=f^{-1}$ may
be expressed as in \cite{z1}:%
\begin{equation}
g(w)=f^{-1}(w)=w+\sum\limits_{n=2}^{\infty}\frac{1}{n}K_{n-1}^{-n}\left(
a_{2},a_{3},...\right)  w^{n}, \label{ieq3}%
\end{equation}
where%
\begin{align}
K_{n-1}^{-n}  &  =\frac{\left(  -n\right)  !}{\left(  -2n+1\right)  !\left(
n-1\right)  !}a_{2}^{n-1}+\frac{\left(  -n\right)  !}{\left(  2\left(
-n+1\right)  \right)  !\left(  n-3\right)  !}a_{2}^{n-3}a_{3}+\frac{\left(
-n\right)  !}{\left(  -2n+3\right)  !\left(  n-4\right)  !}a_{2}^{n-4}%
a_{4}\label{ieq4}\\
&  +\frac{\left(  -n\right)  !}{\left(  2\left(  -n+2\right)  \right)
!\left(  n-5\right)  !}a_{2}^{n-5}\left[  a_{5}+\left(  -n+2\right)  a_{3}%
^{2}\right]  +\frac{\left(  -n\right)  !}{\left(  -2n+5\right)  !\left(
n-6\right)  !}a_{2}^{n-6}
\nonumber\\
&  \left[  a_{6}+\left(  -2n+5\right)  a_{3}a_{4}\right] +\sum\limits_{j\geq7}a_{2}^{n-j}V_{j},\nonumber
\end{align}
such that $V_{j}$ with $7\leq j\leq n$ is a homogeneous polynomial in the
variables $a_{2},a_{3},...,a_{n}$ \cite{z3}.

In particular, the first three terms of $K_{n-1}^{-n}$ are%
\begin{equation}
K_{1}^{-2}=-2a_{2},\text{ }K_{2}^{-3}=3\left(  2a_{2}^{2}-a_{3}\right)
,\text{ }K_{3}^{-4}=-4\left(  5a_{2}^{3}-5a_{2}a_{3}+a_{4}\right).
\label{ieq5}%
\end{equation}

In general, for any $p\in%
%TCIMACRO{\U{2115} }%
%BeginExpansion
\mathbb{N}
%EndExpansion
$ $:=\{1,2,3,...\}$, an expansion of $K_{n}^{p}$ is as in \cite{z1},%
\begin{equation}
K_{n}^{p}=pa_{n}+\frac{p(p-1)}{2}D_{n}^{2}+\frac{p!}{\left(  p-3\right)
!3!}D_{n}^{3}+\cdots+\frac{p!}{\left(  p-n\right)  !n!}D_{n}^{n}, \label{ieq6}%
\end{equation}
where $D_{n}^{p}=D_{n}^{p}\left(  a_{2},a_{3},...\right)  ,$ and by
\cite{z29}, $D_{n}^{m}\left(  a_{1},a_{2},...,a_{n}\right)  =$ $\sum
\limits_{n=1}^{\infty}\frac{m!}{i_{1}!...i_{n}!}a_{1}^{i_{1}}...a_{n}^{i_{n}}$
while $a_{1}=1,$ and the sum is taken over all non-negative integers $i_{1},...,i_{n}$ satisfying $i_{1}%
+i_{2}+\cdots+i_{n}=m,$ $i_{1}+2i_{2}+\cdots+ni_{n}=n,$ it is clear that
$D_{n}^{m}\left(  a_{1},a_{2},...,a_{n}\right)  =a_{1}^{n}.$

Consequently, for functions $f\in\mathfrak{B}_{\Sigma}^{\mu}(\alpha
,\lambda,\delta)$ of the form (\ref{ieq1}), we can write:%
\begin{equation}
(1-\lambda)\left(  \frac{f(z)}{z}\right)  ^{\mu}+\lambda f^{\prime}(z)\left(
\frac{f(z)}{z}\right)  ^{\mu-1}+\xi\delta zf^{\prime\prime}(z)=1+\sum
\limits_{n=2}^{\infty}F_{n-1}\left(  a_{2},a_{3},...,a_{n}\right)  z^{n-1},
\label{ieq7}%
\end{equation}
where%
\[
F_{1}=(\mu+\lambda+2\xi\delta)a_{2},\text{ }F_{2}=(\mu+2\lambda+6\xi
\delta)\left[  \frac{\mu-1}{2}a_{2}^{2}+\left(  1+\frac{6\delta}{2\lambda
+1}\right)  a_{3}\right]  .
\]

In general,%
\begin{align}
F_{n-1}\left(  a_{2},a_{3},...,a_{n}\right)   &  =\left[  \mu+\left(
n-1\right)  \lambda+n\left(  n-1\right)  \xi\delta\right]  \times\left[
\left(  \mu-1\right)  !\right] \label{ieq8}\\
& \ \ \times\left[  \sum\limits_{n=2}^{\infty}\frac{a_{2}^{i_{1}}a_{3}^{i_{2}%
}...a_{n}^{i_{n-1}}}{i_{1}!i_{2}!\cdots i_{n}!\left[  \mu-\left(  i_{1}%
+i_{2}+\cdots+i_{n-1}\right)  \right]  !}\right] \nonumber
\end{align}
is a Faber polynomial of degree $\left(  n-1\right)  .$

Our first theorem introduces an upper bound for the coefficients $\left\vert
a_{n}\right\vert $ of analytic bi-univalent functions in the class
$\mathfrak{B}_{\Sigma}^{\mu}(\alpha,\lambda,\delta).$

\begin{theorem}
\label{thm1} For $\lambda\geq1,$ $\mu\geq0,$ $\delta\geq0$ and $0\leq\alpha<1$,
let the function $f\in$ $\mathfrak{B}_{\Sigma}^{\mu}(\alpha,\lambda,\delta)$ be
given by (\ref{ieq1}). If $a_{k}=0$ $\left(  2\leq k\leq n-1\right)  $, then
\[
\left\vert a_{n}\right\vert \leq\frac{2\left(  1-\alpha\right)  }{\mu+\left(
n-1\right)  \lambda+n\left(  n-1\right)  \xi\delta}\text{ \ }\left(
n\geq4\right)  .
\]

\end{theorem}

\begin{proof}
For the function $f\in$ $\mathfrak{B}_{\Sigma}^{\mu}(\alpha,\lambda,\delta)$ of
the form (\ref{ieq1}), we have the expansion (\ref{ieq7}) and for the inverse
map $g=f^{-1}$, considering (\ref{ieq2}), we obtain
\begin{equation}
(1-\lambda)\left(  \frac{g(w)}{w}\right)  ^{\mu}+\lambda g^{\prime}(w)\left(
\frac{g(w)}{w}\right)  ^{\mu-1}+\xi\delta wg^{\prime\prime}(w)=1+\sum
\limits_{n=2}^{\infty}F_{n-1}\left(  A_{2},A_{3},...,A_{n}\right)  w^{n-1},
\label{ieq9}%
\end{equation}
with%
\begin{equation}
A_{n}=\frac{1}{n}K_{n-1}^{-n}\left(  a_{2},a_{3},...\right)  . \label{ieq10}%
\end{equation}

On the other hand, since $f\in$ $\mathfrak{B}_{\Sigma}^{\mu}(\alpha
,\lambda,\delta)$ and $g=f^{-1}\in$ $\mathfrak{B}_{\Sigma}^{\mu}(\alpha
,\lambda,\delta),$ by definition, there exist two positive real-part functions
$p(z)=1+\sum\limits_{n=1}^{\infty}c_{n}z^{n}\in\mathcal{A}$ and $q(w)=1+\sum
\limits_{n=1}^{\infty}d_{n}w^{n}\in\mathcal{A}$, where $\operatorname{Re}%
\left[  p(z)\right]  >0$ and $\operatorname{Re}\left[  q(w)\right]  >0$ in
$\mathbb{U}$ so that
\begin{align}
&  (1-\lambda)\left(  \frac{f(z)}{z}\right)  ^{\mu}+\lambda f^{\prime
}(z)\left(  \frac{f(z)}{z}\right)  ^{\mu-1}+\xi\delta zf^{\prime\prime
}(z)\label{ieq11}\\
& \qquad =\alpha+\left(  1-\alpha\right)  p(z)=1+\left(  1-\alpha\right)
\sum\limits_{n=1}^{\infty}K_{n}^{1}\left(  c_{1},c_{2},...,c_{n}\right)
z^{n}\nonumber
\end{align}
and%
\begin{align}
& (1-\lambda)\left(  \frac{g(w)}{w}\right)  ^{\mu}+\lambda g^{\prime
}(w)\left(  \frac{g(w)}{w}\right)  ^{\mu-1}+\xi\delta wg^{\prime\prime
}(w)\label{ieq12}\\
& \qquad =\alpha+\left(  1-\alpha\right)  q(w)=1+\left(  1-\alpha\right)
\sum\limits_{n=1}^{\infty}K_{n}^{1}\left(  d_{1},d_{2},...,d_{n}\right)
w^{n}\nonumber
\end{align}

Note that, by the Caratheodory lemma (e.g., \cite{Duren}), $\left\vert
c_{n}\right\vert \leq2$ and $\left\vert d_{n}\right\vert \leq2$ $(n\in%
%TCIMACRO{\U{2115} }%
%BeginExpansion
\mathbb{N}
%EndExpansion
)$. Comparing the corresponding coefficients of (\ref{ieq7}) and
(\ref{ieq11}), for any $n\geq2,$ yields%
\begin{equation}
F_{n-1}\left(  a_{2},a_{3},...,a_{n}\right)  =\left(  1-\alpha\right)
\sum\limits_{n=1}^{\infty}K_{n-1}^{1}\left(  c_{1},c_{2},...,c_{n-1}\right)  ,
\label{ieq13}%
\end{equation}
and similarly, from (\ref{ieq9}) and (\ref{ieq12}) we find%
\begin{equation}
F_{n-1}\left(  A_{2},A_{3},...,A_{n}\right)  =\left(  1-\alpha\right)
\sum\limits_{n=1}^{\infty}K_{n-1}^{1}\left(  d_{1},d_{2},...,d_{n-1}\right)  .
\label{ieq14}%
\end{equation}

Note that for $a_{k}=0$ $\left(  2\leq k\leq n-1\right)  ,$ we have
 \begin{center}
 $A_{n}=-$\ $a_{n}$
\end{center}
  and so $\left[  \mu+\left(  n-1\right)  \lambda\right]
a_{n}=\left(  1-\alpha\right)  c_{n-1}$ and $-\left[  \mu+\left(  n-1\right)  \lambda\right]  a_{n}=\left(  1-\alpha
\right)  d_{n-1}.$

Taking the absolute values of the above equalities, we
obtain%
\begin{align}
\left\vert a_{n}\right\vert &=\frac{\left(  1-\alpha\right)  \left\vert
c_{n-1}\right\vert }{\mu+\left(  n-1\right)  \lambda+n\left(  n-1\right)
\xi\delta} \nonumber \\
&=\frac{\left(  1-\alpha\right)  \left\vert d_{n-1}\right\vert }%
{\mu+\left(  n-1\right)  \lambda+n\left(  n-1\right)  \xi\delta}\leq
\frac{2\left(  1-\alpha\right)  }{\mu+\left(  n-1\right)  \lambda+n\left(
n-1\right)  \xi\delta}, \nonumber
\end{align}
which completes the proof of \ref{thm1}.
\end{proof}

The following corollary is an immediate consequence of the above theorem.

\begin{corollary}
\bigskip\label{cor1} For $\lambda\geq1,$ $\delta\geq0$ and
$0\leq\alpha<1$, let the function $f\in$ $\mathfrak{B}_{\Sigma}(\alpha
,\lambda,\delta)$ be given by (\ref{ieq1}). If $a_{k}=0$ $\left(  2\leq k\leq
n-1\right)  $, then
\[
\left\vert a_{n}\right\vert \leq\frac{2\left(  1-\alpha\right)  }{1+\left(
n-1\right)  \lambda+n\left(  n-1\right)  \xi\delta}\text{ \ }\left(
n\geq4\right).
\]

\end{corollary}

\begin{theorem}
\label{thm2} For $\lambda\geq1,$ $\mu\geq0,$ $\delta\geq0$ and $0\leq\alpha<1$,
let the function $f\in$ $\mathfrak{B}_{\Sigma}^{\mu}(\alpha,\lambda,\delta)$ be
given by (\ref{ieq1}). Then one has the following
\begin{equation}
\hspace{-1in} \left\vert a_{2}\right\vert \leq\left\{
\begin{array}
[c]{c}%
\sqrt{\frac{4\left(  1-\alpha\right)  }{\left(  \mu+2\lambda+6\xi
\delta\right)  \left(  \mu+1\right)  }},\text{ \ \ \ }%
0\leq\alpha\leq\frac{\mu+2\lambda-\lambda^{2}}{\left(  \mu+2\lambda+6\xi
\delta\right)  \left(  \mu+1\right)  }\\
\frac{2\left(  1-\alpha\right)  }{\mu+\lambda+2\xi\delta}\text{
\ \ \ \ \ \ \ \ \ },\text{  \ \ \ }\frac{\mu+2\lambda-\lambda
^{2}}{\left(  \mu+2\lambda+6\xi\delta\right)  \left(  \mu+1\right)  }%
\leq\alpha\leq1
\end{array}
\right.  . \label{ieq15}%
\end{equation}%
\begin{equation}
\hspace{0.7in} \left\vert a_{3}\right\vert \leq\left\{
\begin{array}
[c]{c}%
\min\left\{  \frac{4\left(  1-\alpha\right)  ^{2}}{\left(  \mu+\lambda
+2\xi\delta\right)  ^{2}}+\frac{2\left(  1-\alpha\right)  }{\mu+2\lambda
+6\xi\delta},\frac{4\left(  1-\alpha\right)  }{\left(  \mu+2\lambda+6\xi
\delta\right)  \left(  \mu+1\right)  }\right\}  ,\text{
\ \ \ }0\leq\mu<1\\
\frac{2\left(  1-\alpha\right)  }{\mu+2\lambda+2\xi\delta}\text{
\ \ \ \ \ \ \ \ \ \ \ \ \ \ \ \ \ \ \ \ \ \ \ \ \ \ \ \ \ \ \ \ \ \ \ \ }%
,\text{ \ \ \ \ }\mu\geq1
\end{array}
\right.  \label{ieq16}%
\end{equation}%
and 
\[
\left\vert a_{3}-\frac{\mu+3}{2}a_{2}^{2}\right\vert \leq\frac{2\left(
1-\alpha\right)  }{\mu+2\lambda+6\xi\delta}.
\]

\end{theorem}

\begin{proof}
If we set $n=2$ and $n=3$ in (\ref{ieq13}) and (\ref{ieq14}), respectively, we
get%
\begin{equation}
\left(  \mu+\lambda+2\xi\delta\right)  a_{2}=\left(  1-\alpha\right)  c_{1,}
\label{ieq17}%
\end{equation}%
\begin{equation}
\left(  \mu+2\lambda+6\xi\delta\right)  \left[  \left(  \frac{\mu-1}%
{2}\right)  a_{2}^{2}+\left(  1+\frac{6\delta}{2\lambda+1}\right)
a_{3}\right]  =\left(  1-\alpha\right)  c_{2,} \label{ieq18}%
\end{equation}%
\begin{equation}
-\left(  \mu+\lambda\right)  a_{2}=\left(  1-\alpha\right)  d_{1,}
\label{ieq20}%
\end{equation}%
\begin{equation}
\left(  \mu+2\lambda+6\xi\delta\right)  \left[  \frac{\mu+3}{2}a_{2}%
^{2}-\left(  1+\frac{6\delta}{2\lambda+1}\right)  a_{3}\right]  =\left(
1-\alpha\right)  d_{2,} \label{ieq21}%
\end{equation}

From (\ref{ieq17}) and (\ref{ieq20}), we find (by the Caratheodory lemma)%
\begin{equation}
\left\vert a_{2}\right\vert =\frac{\left(  1-\alpha\right)  \left\vert
c_{1}\right\vert }{\mu+\lambda+2\xi\delta}=\frac{\left(  1-\alpha\right)
\left\vert d_{1}\right\vert }{\mu+\lambda+2\xi\delta}\leq\frac{2\left(
1-\alpha\right)  }{\mu+\lambda+2\xi\delta}. \label{ieq22}%
\end{equation}

Also from (\ref{ieq18}) and (\ref{ieq21}), we obtain%
\begin{equation}
\left(  \mu+2\lambda+6\xi\delta\right)  \left(  \mu+1\right)  a_{2}%
^{2}=\left(  1-\alpha\right)  \left(  c_{2}+d_{2}\right)  . \label{ieq25}%
\end{equation}

Using the Caratheodory lemma, we get $\left\vert a_{2}\right\vert \leq
\sqrt{\frac{4\left(  1-\alpha\right)  }{\left(  \mu+2\lambda+6\xi
\delta\right)  \left(  \mu+1\right)  }},$ and combining this with inequality
(\ref{ieq22}), we obtain the desired estimate on the coefficient $\left\vert
a_{2}\right\vert $ as asserted in (\ref{ieq15}).

Next, in order to find the bound on the coefficient $\left\vert a_{3}%
\right\vert $, we subtract (\ref{ieq21}) from (\ref{ieq18}).

We thus get
\[
\left(  \mu+2\lambda+6\xi\delta\right)  \left(  -2a_{2}^{2}+2\left(
1+\frac{6\delta}{2\lambda+1}\right)  a_{3}\right)  =\left(  1-\alpha\right)
\left(  c_{2}-d_{2}\right).
\]
or%
\begin{equation}
a_{3}=a_{2}^{2}+\frac{\left(  1-\alpha\right)  \left(  c_{2}-d_{2}\right)
}{2\left(  \mu+2\lambda+6\xi\delta\right)  }. \label{ieq26}%
\end{equation}

Upon substituting the value of $a_{2}^{2}$ from (\ref{ieq17}) into
(\ref{ieq26}), it follows that%
\[
a_{3}=\frac{\left(  1-\alpha\right)  ^{2}c_{1}^{2}}{\left(  \mu+\lambda
+2\xi\delta\right)  ^{2}}+\frac{\left(  1-\alpha\right)  \left(  c_{2}%
-d_{2}\right)  }{2\left(  \mu+2\lambda+6\xi\delta\right)  }.
\]

We thus find (by the Caratheodory lemma) that%
\begin{equation}
\left\vert a_{3}\right\vert \leq\frac{4\left(  1-\alpha\right)  ^{2}}{\left(
\mu+\lambda+2\xi\delta\right)  ^{2}}+\frac{2\left(  1-\alpha\right)  }{\left(
\mu+2\lambda+6\xi\delta\right)  }. \label{ieq27}%
\end{equation}

On the other hand, upon substituting the value of $a_{2}^{2}$ from
(\ref{ieq25}) into (\ref{ieq26}), it follows that%
\[
a_{3}=\frac{\left(  1-\alpha\right)  }{2\left(  \mu+2\lambda+6\xi
\delta\right)  \left(  \mu+1\right)  }\left[  \left(  \mu+3\right)
c_{2}+\left(  1-\mu\right)  d_{2}\right]  .
\]

Consequently (by the Caratheodory lemma), we have%
\begin{equation}
\left\vert a_{3}\right\vert \leq\frac{\left(  1-\alpha\right)  }{\left(
\mu+2\lambda+6\xi\delta\right)  \left(  \mu+1\right)  }\left[  \left(
\mu+3\right)  +\left\vert 1-\mu\right\vert \right]  . \label{ieq28}%
\end{equation}

Combining (\ref{ieq27}) and (\ref{ieq28}), we get the desired estimate on the
coefficient $\left\vert a_{3}\right\vert $ as asserted in (\ref{ieq16}).
Finally, from (\ref{ieq21}), we deduce (by the Caratheodory lemma) that%
\[
\left\vert a_{3}-\frac{\mu+3}{2}a_{2}^{2}\right\vert =\frac{\left(
1-\alpha\right)  \left\vert d_{2}\right\vert }{\mu+2\lambda+6\xi\delta}%
\leq\frac{2\left(  1-\alpha\right)  }{\mu+2\lambda+6\xi\delta}.
\]

This evidently completes the proof of \ref{thm2}.
\end{proof}

By setting $%
%TCIMACRO{\U{3bc} }%
%BeginExpansion
\mu
%EndExpansion
=1$ in \ref{thm2}, we obtain the following consequence.

\begin{corollary}
\label{cor2} For $\lambda\geq1,$ $\delta\geq0$ and $0\leq
\alpha<1$, let the function $f\in$ $\mathfrak{B}_{\Sigma}(\alpha,\lambda,\delta)$
be given by (\ref{ieq1}). Then one has the following%
\[
\left\vert a_{2}\right\vert \leq\left\{
\begin{array}
[c]{c}%
\sqrt{\frac{2\left(  1-\alpha\right)  }{\left(  1+2\lambda+6\xi\delta\right)
}},\text{  \ \ \ }0\leq\alpha\leq\frac{1+2\lambda-\lambda
^{2}}{2\left(  1+2\lambda+6\xi\delta\right)  }\\
\frac{2\left(  1-\alpha\right)  }{1+\lambda+2\xi\delta}\text{
\ \ \ \ \  },\text{  \ \ \ }\frac{1+2\lambda-\lambda^{2}%
}{2\left(  1+2\lambda+6\xi\delta\right)  }\leq\alpha\leq1
\end{array}
\right.
\]%
\[
\hspace{-1.8in} \left\vert a_{3}\right\vert \leq\frac{2\left(  1-\alpha\right)  }%
{1+2\lambda+6\xi\delta},%
\]%
and
\[
\left\vert a_{3}-2a_{2}^{2}\right\vert \leq\frac{2\left(  1-\alpha\right)
}{1+2\lambda+6\xi\delta}.
\]

\end{corollary}

\newpage
By setting $\lambda=1$ in Theorem \ref{thm2}, we obtain the following consequence.

\begin{corollary}
\label{cor3} For $\mu\geq0,$ $\delta\geq0$ and $0\leq\alpha<1$, let the
function $f\in$ $\mathfrak{B}_{\Sigma}^{\mu}(\alpha,\delta)$ be given by
(\ref{ieq1}). Then one has the following%
\[
\hspace{-1.4in} \ \left\vert a_{2}\right\vert \leq\left\{
\begin{array}
[c]{c}%
\sqrt{\frac{4\left(  1-\alpha\right)  }{\left(  \mu+6\xi\delta+2\right)
\left(  \mu+1\right)  }},\text{ \ \ \ \ }0\leq\alpha\leq
\frac{1}{\mu+6\xi\delta+2}\\
\frac{2\left(  1-\alpha\right)  }{\mu+2\xi\delta+1}\text{ \ \ \ \ \ \ \ \ \ }%
,\text{ \ \ \ }\frac{1}{\mu+6\xi\delta+2}\leq\alpha\leq1
\end{array}
\right.
\]%
\[
\left\vert a_{3}\right\vert \leq\left\{
\begin{array}
[c]{c}%
\min\left\{  \frac{4\left(  1-\alpha\right)  ^{2}}{\left(  \mu+2\xi
\delta+1\right)  ^{2}}+\frac{2\left(  1-\alpha\right)  }{\mu+6\xi\delta
+2},\frac{4\left(  1-\alpha\right)  }{\left(  \mu+6\xi\delta+2\right)  \left(
\mu+1\right)  }\right\}  ,\text{ \ \ \  }0\leq\mu<1\\
\frac{2\left(  1-\alpha\right)  }{\mu+6\xi\delta+2}\text{
\ \ \ \ \ \ \ \ \ \ \ \ \ \ \ \ \ \ \ \ \ \ \ \ \ \ \ \ \ \ \ \ \ \ \ \ }%
,\text{ \ \ \ \ }\mu\geq1
\end{array}
\right.
\]

\end{corollary}

\bigskip
\begin{remark}
\label{rem2} As a final remark, for $\delta=0$ in 

(i) Theorem \ref{thm1} we obtain Theorem 1 in \cite{bul}. \vspace{0.05in}

(ii) Theorem \ref{thm2} we obtain Theorem 2 in \cite{bul}. \vspace{0.05in}

(iii) Corollary \ref{cor1} we obtain Theorem 1 in \cite{19}. \vspace{0.05in}

(iv) Corollary \ref{cor2} we obtain Theorem 2 in \cite{19}. \vspace{0.05in}

(v) Corollary \ref{cor3} we obtain Corollary 3 in \cite{bul}.
\end{remark}

\end{document}